\documentclass[12pt,a4paper,reqno]{amsart}
\usepackage{amsmath,amsthm,amssymb,mathtools,enumitem,setspace}
\usepackage{a4wide}
\usepackage[margin=2cm]{geometry}

\usepackage[pagebackref,colorlinks=true,citecolor=blue,linkcolor=red,urlcolor=blue]{hyperref}

\title[Non-commuting, non-generating graphs of nilpotent groups]{The non-commuting, non-generating graph\\of a nilpotent group}

\author{Peter J. Cameron, Saul D. Freedman and Colva M. Roney-Dougal}
\date{\today}
\subjclass[2010]{20F18, 05C25}
\thanks{\textit{\keywordsname}. $p$-group, nilpotent group, generating graph, commuting graph}
\address{ %
School of Mathematics and Statistics, University of St Andrews, St Andrews, KY16 9SS, UK}
\email{pjc20@st-andrews.ac.uk}
\email{sdf8@st-andrews.ac.uk}
\email{colva.roney-dougal@st-andrews.ac.uk}

\newtheorem{thm}{Theorem}[section]
\newtheorem*{thm*}{Theorem}
\newtheorem*{conj*}{Conjecture}
\newtheorem{cor}[thm]{Corollary}

\newtheorem{lem}[thm]{Lemma}

\newtheorem{prop}[thm]{Proposition}

\theoremstyle{remark}

\theoremstyle{definition}

\newtheorem{rem}[thm]{Remark}

\numberwithin{equation}{section}

\newcommand\diam{\mathrm{diam}}

\newcommand\nc{\Xi}

\newcommand\nd{\Xi^+}

\begin{document}

\onehalfspacing

\begin{abstract}
For a nilpotent group $G$, let $\nc(G)$ be the difference between the complement of the generating graph of $G$ and the commuting graph of $G$, with vertices corresponding to central elements of $G$ removed. That is, $\nc(G)$ has vertex set $G \setminus Z(G)$, with two vertices adjacent if and only if they do not commute and do not generate $G$. Additionally, let $\nd(G)$ be the subgraph of $\nc(G)$ induced by its non-isolated vertices. We show that if $\nc(G)$ has an edge, then $\nd(G)$ is connected with diameter $2$ or $3$, with $\nc(G) = \nd(G)$ in the diameter $3$ case. In the infinite case, our results apply more generally, to any group with every maximal subgroup normal. When $G$ is finite, we explore the relationship between the structures of $G$ and $\nc(G)$ in more detail.
\end{abstract}

\maketitle

\section{Introduction}
\label{sec:intro}

A number of graphs have been defined on the set of elements of a group, aiming to capture some aspect of the group structure in graph-theoretic terms. Prominent among these is the \emph{commuting graph}, in which two elements are joined by an edge if and only if they commute. This graph is trivially connected with diameter at most $2$, since the identity is connected to all other vertices. However, if we remove the central elements of the group, then the commuting graph is no longer obviously connected, and indeed it fails to be connected for some groups. Giudici and Parker \cite{giudiciparker} showed that, even if it is connected, its diameter may be unbounded. On the other hand, Morgan and Parker \cite{morganparker} showed that if a group has trivial centre, then any connected component of its commuting graph has diameter at most $10$.

The commuting graph of a group $G$ fits into a hierarchy of graphs as follows. Each graph now has vertex set $G\setminus\{1\}$, and the rules for adjacency of vertices $x$ and $y$ are as follows:
\begin{itemize}
\item the \emph{power graph}: one of $x$ and $y$ is a power of the other;
\item the \emph{enhanced power graph}: $x$ and $y$ are powers of some element $z$;
\item the \emph{commuting graph}: $[x,y] = 1$;
\item the \emph{non-generating graph}: $\langle x,y\rangle\ne G$;
\item the \emph{complete graph}: all pairs are adjacent.
\end{itemize}
Observe that each graph is a spanning subgraph of the next (except for the third and fourth when $G$ is $2$-generated and abelian). This allows us to refine questions about connectedness, and ask, for instance, whether the difference between consecutive graphs in the hierarchy is connected (often with specified vertices removed). For example, the difference between the complete and non-generating graphs of $G$ is the \emph{generating graph} $\Gamma(G)$, where $x$ and $y$ are adjacent if and only if $\langle x, y \rangle = G$. For convenience, we will write $\Gamma^+(G)$ to denote the subgraph of $\Gamma(G)$ induced by its non-isolated vertices.


The most basic question is whether the difference between consecutive graphs has any edges at all, for $|G| > 2$. Finite groups whose power graph and enhanced power graph are equal, or whose enhanced power graph and commuting graph are equal, were determined by Aalipour \emph{et al}.~\cite{aalipour}. The commuting graph and the non-generating graph of a non-abelian group are equal if and only if the group is \emph{minimal non-abelian}, i.e., every proper subgroup of the group is abelian. Finally, $\Gamma(G)$ has no edges if and only if $G$ is not $2$-generated.

More complicated questions regarding connectedness have also been asked and resolved in the case of $\Gamma(G)$. For example, Breuer, Guralnick and Kantor \cite[Theorem 1.2]{breuer} proved that if $G$ is a non-abelian finite simple group, then $\Gamma(G)$ is extremely ``dense'', in the sense that it is connected with diameter $2$. More recently, Burness, Guralnick and Harper \cite[Corollary 6]{burness} generalised this result: if $G$ is a finite group, then $\Gamma(G)$ either has an isolated vertex or is connected with diameter at most $2$.

The structure of $\Gamma(G)$ is less predictable, however, when isolated vertices are involved. For example, Crestani and Lucchini \cite[Theorem 1.3]{crestani} showed that although $\Gamma^+(G)$ is connected when $G$ is a $2$-generated direct power of $\mathrm{SL}(2,2^p)$, with $p$ an odd prime, $G$ can be chosen so that the diameter of $\Gamma^+(G)$ is arbitrarily large. On the other hand, Lucchini \cite[Corollary 4]{lucchinisoluble} proved that if $G$ is finite and nilpotent, then $\Gamma^+(G)$ is connected with diameter at most $2$.

This motivates the present paper, which considers the \emph{non-commuting, non-generating graph} $\nc(G)$ of a group $G$, i.e., the difference between the non-generating graph and the commuting graph of $G$. In this case, we remove all vertices corresponding to central elements of $G$, as otherwise these would always be isolated. Thus $\nc(G)$ is the graph whose vertices are $G \setminus Z(G)$, with vertices $x$ and $y$ adjacent if and only if $[x,y] \ne 1$ and $\langle x, y \rangle \ne G$. In addition, we write $\nd(G)$ to denote the subgraph of $\nc(G)$ induced by its non-isolated vertices.

We will focus in this paper on the case where $G$ is a (not necessarily finite) nilpotent group, and describe the possible diameters of the connected components of $\nc(G)$. As the unique nontrivial connected component of the generating graph of a finite $2$-generated nilpotent group $G$ is always extremely dense, one may expect that the same is not necessarily true for $\nc(G)$, which is a proper subgraph of the complement of this generating graph. However, the first of our main theorems shows that this is actually the case for any nilpotent group. 

\begin{thm}
\label{thm:nilpncsummary}
Let $G$ be a nilpotent group. If $\nc(G)$ contains an edge, then $\nd(G)$ is connected with diameter $2$ or $3$. Moreover, if $\diam(\nd(G)) = 3$, then $\nc(G) = \nd(G)$.
\end{thm}

We will in fact prove that the conclusion of this theorem holds whenever $G$ is a group with every maximal subgroup normal. This is a weaker condition than nilpotency, as illustrated by the infinite $2$-generated $3$-group constructed by Gupta and Sidki \cite{guptasidki}. Indeed, every maximal subgroup of this group is normal \cite[Theorem 4.3]{pervova}; however, the group has no finite presentation \cite{sidki}, and is therefore not nilpotent. For further discussion of groups with every maximal subgroup normal, see, for example \cite{myropolskaarxiv,myropolska}.

\bigskip

It is clear that, for an arbitrary group $G$, the graph $\nc(G)$ is empty if and only if $G$ is abelian. As mentioned above, $\nc(G)$ has vertices but no edges if and only if $G$ is minimal non-abelian. Since a non-central element of a group is centralised by at most one maximal subgroup, a minimal non-abelian group is necessarily $2$-generated.

The finite minimal non-abelian groups were classified by Miller and Moreno \cite{miller} in 1903 (see also \cite{redei}). In particular, any such group is either a $p$-group, for some prime $p$, or a non-nilpotent group whose order is divisible by exactly two primes. A concise description of the finite minimal non-abelian $p$-groups is given in \cite[Theorem 2.4]{zhangsun}. On the other hand, a classification of infinite minimal non-abelian groups is not complete. Well-known examples are the \emph{Tarski monsters}, which are simple groups where every proper nontrivial subgroup is cyclic of fixed prime order $p$. Ol'shanski\u{\i} \cite{olshanskii} proved that a Tarski monster exists for each prime $p > 10^{75}$.

%
%

The relationship between $G$ and $\nc(G)$ is less clear when $\nc(G)$ has edges. The remaining two main theorems of this paper provide a detailed overview of this relationship in the case where $G$ is a finite nilpotent group. We let $\Phi(G)$ denote the Frattini subgroup of $G$.

\begin{thm}
\label{thm:pgpnc}
Let $G$ be a finite $p$-group. Then one of the following occurs.
\begin{enumerate}[label={(\roman*)},font=\upshape]
\item $G$ is either abelian or minimal non-abelian. In this case, $\nc(G)$ has no edges.
\item $G$ is non-abelian and not $2$-generated. In this case, $\nc(G)$ is connected with diameter $2$.
\item $G$ is non-abelian, $2$-generated and not minimal non-abelian, and contains at most one abelian maximal subgroup. Furthermore, each maximal subgroup contains $Z(G)$.
\begin{enumerate}[label={(\alph*)},font=\upshape]
\item If $G$ has an abelian maximal subgroup $M$, then $\nd(G)$ is connected with diameter $2$, and the isolated vertices of $\nc(G)$ are precisely the elements of $M \setminus \Phi(G)$.
\item If the centre of each maximal subgroup of $G$ is equal to $Z(G)$, then $\nc(G)$ is connected with diameter $2$.
\item If all maximal subgroups of $G$ are non-abelian, and at least one has a centre properly containing $Z(G)$, then $\nc(G)$ is connected with diameter $3$.
\end{enumerate}
\end{enumerate}
\end{thm}

Using Magma \cite{magma}, we can show that the groups numbered $(16,7)$, $(243,3)$ and $(32,6)$ in the Small Groups Library \cite{smallgroups} are groups of the smallest order satisfying (iii)(a), (b) and (c), respectively.


\begin{thm}
\label{thm:nilpgpnc}
Let $G$ be a finite nilpotent group whose order is divisible by at least two primes. Then one of the following occurs.
\begin{enumerate}[label={(\roman*)},font=\upshape]
\item $G$ is abelian. In this case, $\nc(G)$ is the empty graph.
\item $G$ is non-abelian and contains at least two non-cyclic Sylow subgroups. In this case, $\nc(G)$ is connected with diameter $2$.
\item $G = P \times H$, with $P$ a non-abelian Sylow subgroup of $G$ and $H$ cyclic.
\begin{enumerate}[label={(\alph*)},font=\upshape]
\item If $\nc(P)$ contains no isolated vertices, then $\nc(G)$ is connected, and $\diam(\nc(G)) = \diam(\nc(P)) \in \{2,3\}$.
\item If $\nc(P)$ contains isolated vertices, then $\nd(G)$ is connected with diameter $2$, and $(g,h) \in G$ is an isolated vertex of $\nc(G)$ if and only if $g$ is an isolated vertex of $\nc(P)$ and $h$ is a generator of $H$.
\end{enumerate}
\end{enumerate}
\end{thm}

Note that the group $P$ in case (iii)(b) of the above theorem may be minimal non-abelian.

The structure of this paper is as follows. In \S\ref{sec:prelim}, we prove several preliminary results about the graph $\nc(G)$. Next, in \S\ref{sec:normalmax}, we derive useful properties of groups that contain a normal maximal subgroup. This section concludes with a proof of Theorem \ref{thm:nilpncsummary}. We then prove Theorem \ref{thm:pgpnc} in \S\ref{sec:pgroups}. Finally, in \S\ref{sec:dirprods}, we investigate the structure of $\nc(G)$ when $G$ is a direct product of groups, and in particular prove Theorem \ref{thm:nilpgpnc}.

\section{Preliminary results}
\label{sec:prelim}

We now prove several elementary but useful results about the non-commuting, non-generating graph of an arbitrary group $G$. Note that if $x$ and $y$ are vertices of a graph, then we write $d(x,y)$ to denote the distance in the graph between $x$ and $y$.

\begin{prop}
\label{prop:nodiamonenc}
No connected component of $\nc(G)$ has diameter $1$.
\end{prop}

\begin{proof}
Suppose, for a contradiction, that $\nc(G)$ has a connected component $X$ of diameter $1$. Observe that an element of $G \setminus Z(G)$ and its inverse have the same set of neighbours in $\nc(G)$, and are not adjacent. This implies that $|h| = 2$ for each $h \in X$.

Let $x, y \in X$. Then $\{x,y\}$ is an edge of $\nc(G)$, and it follows easily that $\{x,xy\}$ is also an edge. Hence $xy \in X$. Thus $x$, $y$ and $xy$ all have order $2$, and we conclude that $[x,y] = (xy)^2 = 1$. Therefore, $\{x,y\}$ is not an edge of $\nc(G)$, a contradiction.
\end{proof}

In order to deduce additional properties of $\nc(G)$, we will briefly consider the \emph{non-commuting graph} of $G$. This graph has vertex set $G \setminus Z(G)$, with two vertices adjacent if and only if they do not commute. Thus the graph is the complement of the commuting graph (with central elements removed). Abdollahi, Akbari and Maimani proved the following result, using similar arguments to those in the proof of Proposition \ref{prop:nodiamonenc} above, as well as the fact that the union of two proper subgroups of $G$ is a proper subset of $G$.

\begin{prop}[{\cite[Proposition 2.1]{abdollahi}}]
\label{prop:noncomdiam}
Suppose that $G$ is non-abelian. Then the non-commuting graph of $G$ is connected with diameter $2$.
\end{prop}


We are now able to derive two useful corollaries involving $\nc(G)$.

\begin{cor}
\label{cor:propernc}
Let $H$ be a proper non-abelian subgroup of $G$. Then the induced subgraph of $\nc(G)$ corresponding to $H \setminus Z(H)$ is connected with diameter $2$.
\end{cor}

\begin{proof}
Since any two elements of $H$ generate a proper subgroup of $G$, the induced subgraph of $\nc(G)$ corresponding to $H \setminus Z(H)$ is the non-commuting graph of $H$. Hence Proposition \ref{prop:noncomdiam} yields the result.
\end{proof}

Next, notice that if $G$ is not $2$-generated, then $\nc(G)$ is equal to the non-commuting graph of $G$. We therefore obtain the following corollary of Proposition \ref{prop:noncomdiam}.

\begin{cor}
\label{cor:highgengp}
Suppose that $G$ is non-abelian and not $2$-generated. Then $\nc(G)$ is connected with diameter $2$.
\end{cor}

Hence to prove our main theorems, we need only consider non-abelian $2$-generated groups that are not minimal non-abelian.

\begin{rem}
Suppose that $G$ is finite, non-abelian and nilpotent. The non-commuting, non-generating graph of the direct product of $G$ and a finite cyclic group is described in Theorem \ref{thm:nilpgpnc}. We can also determine the structure of $\nc(G \times \mathbb{Z})$. Observe that $G$ contains a non-cyclic Sylow subgroup $P$. Let $k$ be the smallest size of a generating set for $P$. Then $k > 1$, and Burnside's Basis Theorem implies that $P/\Phi(P)$ is isomorphic to the elementary abelian group $C_p^k$ for some prime $p$. As $C_p^{k+1}$ is a quotient of $P/\Phi(P) \times \mathbb{Z}$, which is a quotient of $G \times \mathbb{Z}$, we conclude that the smallest size of a generating set for $G \times \mathbb{Z}$ is at least $k+1 > 2$. Hence Corollary \ref{cor:highgengp} shows that $\nc(G \times \mathbb{Z})$ is connected with diameter $2$.
\end{rem}

The final result in this section generalises the observation that every vertex of $\nc(G)$ is isolated if and only if $G$ is minimal non-abelian. 

\begin{prop}
\label{prop:isolvertnc}
Suppose that $G$ is non-abelian and $2$-generated, and let $g \in G \setminus Z(G)$. Then $g$ is an isolated vertex of $\nc(G)$ if and only if:
\begin{enumerate}[label={(\roman*)},font=\upshape]
\item there exists a unique maximal subgroup $M$ of $G$ containing $g$; and
\item $g \in Z(M)$.
\end{enumerate}
Moreover, if $g$ is not isolated in $\nc(G)$, then there exists a maximal subgroup $L$ of $G$ such that $g \in L \setminus Z(L)$.
\end{prop}

\begin{proof}
An element $g \in G \setminus Z(G)$ is isolated in $\nc(G)$ if and only if $[g,x] = 1$ for each $x \in G$ with $\langle g, x \rangle < G$. As $G$ is finitely generated, each of its proper subgroups lies in some maximal subgroup. Thus $g$ is isolated if and only if each maximal subgroup containing $g$ also centralises $g$. Since $g$ is centralised by at most one maximal subgroup, either (i) and (ii) both hold (and $g$ is isolated) or $g$ is non-central in at least one maximal subgroup (and is not isolated).
%
%
%
\end{proof}

Note that if $G$ is finitely generated but not $2$-generated, then each element of $G$ lies in more than one maximal subgroup.

\section{Groups with normal maximal subgroups}
\label{sec:normalmax}

In this section, we derive some of the most important tools that we will use to prove our main theorems. We will conclude the section with a proof of Theorem \ref{thm:nilpncsummary}.


\begin{prop}
\label{prop:noabelmax}
Suppose that $G$ contains a normal non-abelian maximal subgroup $M$, with $Z(G) < Z(M)$. Then each maximal subgroup of $G$ is non-abelian.
\end{prop}

\begin{proof}
Assume, for a contradiction, that $G$ contains an abelian maximal subgroup $L$. It is clear that $L \cap M$ contains $Z(G)$. Let $x \in Z(M) \setminus Z(G)$. Then $M = C_G(x)$, and so $x \notin L$. Thus there is no element $y \in (L \cap M) \setminus Z(G)$, as otherwise $Z(M)$ would lie in $C_G(y) = L$. Therefore, $L \cap M$ is equal to $Z(G) = L \cap Z(M)$. It follows that $$G/Z(M) = LZ(M)/Z(M) \cong L/(L \cap Z(M)) = L/(L \cap M) \cong LM/M = G/M.$$ This is a contradiction, as $G/M$ is simple, while $G/Z(M)$ is not.
\end{proof}

Our next result, together with Proposition \ref{prop:isolvertnc}, allows us to better characterise isolated vertices of $\nc(G)$ in the case where every maximal subgroup of $G$ is normal.

\begin{prop}
\label{prop:normisol}
Suppose that $G$ is $2$-generated, and let $g$ be an element of $G$ that lies in a unique maximal subgroup $M$ of $G$. Assume in addition that $M \trianglelefteq G$ and $g \in Z(M)$. Then $M$ is abelian.
\end{prop}

\begin{proof}
Let $h \in G \setminus M$. Then $\langle g, h \rangle$ lies in no maximal subgroup of $G$, and hence in no proper subgroup, since $G$ is finitely generated. It follows that $\langle Z(M), h \rangle = G$, and so $G/Z(M)$ is cyclic. Thus $M/Z(M)$ is also cyclic, and therefore $M$ is abelian.
\end{proof}

Our next result provides upper bounds (and in some cases exact values) for distances between certain non-isolated vertices of $\nc(G)$, and in fact between any two non-isolated vertices when every maximal subgroup of $G$ is normal.

\begin{lem}
\label{lem:ncmaxnorm}
Suppose that $G$ is $2$-generated. In addition, let $(x,L,y,M)$ be an ordered $4$-tuple such that $L$ and $M$ are normal, non-abelian maximal subgroups of $G$, with $x \in L \setminus Z(L)$ and $y \in M \setminus Z(M)$. Then $d(x,y) \le 3$. Moreover, $d(x,y) = 3$ if and only if either:
\begin{enumerate}[label={(\roman*)},font=\upshape]
\item $x \in Z(M)$, $y \notin L$, and $M$ is the only maximal subgroup of $G$ containing but not centralising $y$; or
\item $y \in Z(L)$, $x \notin M$, and $L$ is the only maximal subgroup of $G$ containing but not centralising $x$.
\end{enumerate}
\end{lem}

\begin{proof} As $G$ is finitely generated, each proper subgroup of $G$ lies in a maximal subgroup. We split the proof into several cases, corresponding to where $x$ lies with respect to $M$ and $Z(M)$, and where $y$ lies with respect to $L$ and $Z(L)$.

\medskip

\noindent \textbf{Case (a)}: $x \in M \setminus Z(M)$ or $y \in L \setminus Z(L)$. Here, Corollary \ref{cor:propernc} yields $d(x,y) \le 2$.

\medskip

\noindent \textbf{Case (b)}: $x \notin M$ and $y \notin L$. As $L$ and $M$ are normal and maximal in $G$, their intersection $H:=L \cap M$ is maximal in $L$ and in $M$. Notice that $C_H(x)$ is a proper subgroup of $H$, as otherwise $\langle H, x \rangle = L$ would centralise $x$, contradicting $x \notin Z(L)$. Similarly, $C_H(y) < H$. Since the union of two proper subgroups of $H$ is a proper subset of $H$, there exists an element $k \in H$ that centralises neither $x$ nor $y$. Then $(x,k,y)$ is a path in $\nc(G)$, and hence $d(x,y) \le 2$.

\medskip

\noindent \textbf{Case (c)}: $x \in Z(M)$ and $y \in Z(L)$. As $L$ is non-abelian, the quotient $G/Z(L)$ is not cyclic. Hence no element of $Z(L)$ lies in a generating set for $G$ of size two. Thus $y \in Z(L)$ is adjacent in $\nc(G)$ to each element of $G$ that does not centralise $y$, i.e., each element of $G \setminus L$. Similarly, $x$ is adjacent to each element of $G \setminus M$, and so $d(x,y) \le 2$.

\medskip

\noindent \textbf{Case (d)}: $x \in Z(M)$ and $y \notin L$, or $x \notin M$ and $y \in Z(L)$. As there is complete symmetry in $(y,M)$ and $(x,L)$, we may assume that $x \in Z(M)$ and $y \notin L$. Then $x \in (L \cap M) \setminus Z(L)$, and thus $Z(L) \cap M < L \cap M$. Additionally, as in the proof of Case (b), $C_{L \cap M}(y) < L \cap M$. Thus there exists an element $t \in L \cap M$ that does not lie in $Z(L)$ or $C_G(y)$. Then $\{t,y\}$ is an edge of $\nc(G)$, and $d(x,t) \le 2$ by Corollary \ref{cor:propernc}. Hence $d(x,y) \le 3$. It remains to show that $d(x,y) = 3$ if and only if $M$ is the unique maximal subgroup of $G$ that contains but does not centralise $y$.

If $M$ is indeed the unique maximal subgroup of $G$ that contains but does not centralise $y$, then the neighbourhood of $y$ in $\nc(G)$ is a subset of $M$. However, since $x \in Z(M)$, no element of $M$ is a neighbour of $x$. Thus $d(x,y) > 2$, and so $d(x,y) = 3$.

Suppose instead that there exists a maximal subgroup $K$ of $G$ that contains but does not centralise $y$, with $K \ne M$. Then $K \cap M$ and $C_K(y)$ are proper subgroups of $K$, and so there exists an element $s \in K$ that does not lie in $M$ or $C_G(y)$. It is clear that $\{s,y\}$ is an edge of $\nc(G)$. In addition, since $x \in Z(M)$, we conclude as in the proof of Case (c) that $x$ is adjacent in $\nc(G)$ to each element of $G \setminus M$, and in particular to $s$. Thus $(x,s,y)$ is a path in $\nc(G)$, and hence $d(x,y) \le 2$.
\end{proof}

We are now able to prove a more general version of Theorem \ref{thm:nilpncsummary}.

\begin{thm}
\label{thm:nilpncgeneral}
Let $G$ be a group with every maximal subgroup normal. If $\nc(G)$ contains an edge, then $\nd(G)$ is connected with diameter $2$ or $3$. Moreover, if $\diam(\nd(G)) = 3$, then $\nc(G) = \nd(G)$.
\end{thm}

\begin{proof}
Suppose that $\nc(G)$ contains an edge, so that $G$ is non-abelian and not minimal non-abelian. We may also assume that $G$ is $2$-generated, as otherwise $\nc(G)$ is connected with diameter $2$ by Corollary \ref{cor:highgengp}.

Let $x$ and $y$ be non-isolated vertices of $\nc(G)$. Then Proposition \ref{prop:isolvertnc} shows that $G$ contains maximal subgroups $L$ and $M$ such that $x \in L \setminus Z(L)$ and $y \in M \setminus Z(M)$. As these maximal subgroups are normal in $G$, applying Lemma \ref{lem:ncmaxnorm} to the $4$-tuple $(x,L,y,M)$ gives $d(x,y) \le 3$. The theorem's first conclusion now follows from Proposition \ref{prop:nodiamonenc}, which shows that no connected component of $\nc(G)$ has diameter $1$.

Suppose now that $\nc(G)$ has an isolated vertex $g$. To complete the proof, we will show that $\diam(\nd(G)) = 2$. We observe from Proposition \ref{prop:isolvertnc} that $g$ lies in a unique maximal subgroup $K$ of $G$, and that $g \in Z(K)$. Moreover, as $K \trianglelefteq G$, Proposition \ref{prop:normisol} shows that $K$ is abelian. It therefore follows from Proposition \ref{prop:noabelmax} that each non-abelian maximal subgroup $M$ of $G$ satisfies $Z(M) \le Z(G)$. Thus Lemma \ref{lem:ncmaxnorm} yields $\diam(\nd(G)) = 2$.
\end{proof}

Theorem \ref{thm:nilpncsummary} now follows immediately.

\section{Finite $p$-groups}
\label{sec:pgroups}

In this section, we determine the diameters of the connected components of $\nc(G)$ when $G$ is a finite $p$-group, with $p$ a prime, and in particular prove Theorem \ref{thm:pgpnc}. Corollary \ref{cor:highgengp} implies that it suffices to consider the case where $G$ is non-abelian and $2$-generated. We therefore begin by deducing information about non-abelian, finite, $2$-generated $p$-groups.

The ``equality'' statement in part (iii) of the following lemma is well known (for example, see \cite[Lemma 2.3]{zhangsun}).

\begin{lem}
\label{lem:zinfrat}
Let $G$ be a non-abelian, finite, $2$-generated $p$-group. Then the following statements hold.
\begin{enumerate}[label={(\roman*)},font=\upshape]
\item \label{zinfrat1} $\Phi(G)$ is a maximal subgroup of each maximal subgroup of $G$.
\item \label{zinfrat2} Each element of $G \setminus \Phi(G)$ lies in a unique maximal subgroup of $G$.
\item \label{zinfrat3} $Z(G) \le \Phi(G)$, with equality if and only if $G$ is minimal non-abelian.
\item \label{zinfrat4} If $G$ is not minimal non-abelian, then $G$ contains at most one abelian maximal subgroup.
\end{enumerate}
\end{lem}

\begin{proof}\leavevmode

\noindent (i) As $G$ is non-abelian and $2$-generated, Burnside's Basis Theorem implies that $\Phi(G)$ has index $p^2$ in $G$. Hence $\Phi(G)$ is maximal in each maximal subgroup of $G$.

\medskip

\noindent (ii) By (i), $\Phi(G)$ is the intersection of each pair of distinct maximal subgroups of $G$.

\medskip

\noindent (iii)--(iv) Let $M$ be a maximal subgroup of $G$, and let $x \in M \setminus \Phi(G)$. Then by (ii), $M$ is the unique maximal subgroup of $G$ containing $x$. As the abelian group $\langle x, Z(G) \rangle$ is a proper subgroup of $G$, it follows that $M$ contains $Z(G)$. As this holds for every maximal subgroup of $G$, we conclude that $Z(G) \le \Phi(G)$.

Now, if $Z(G) = \Phi(G)$, then $\langle \Phi(G),x \rangle$ is abelian, and is equal to $M$ by (i). This again holds for every maximal subgroup of $G$, and so $G$ is minimal non-abelian.

If instead there exists $y \in \Phi(G) \setminus Z(G)$, then at most one maximal subgroup of $G$ centralises $y$. As $y$ lies in each maximal subgroup of $G$, it follows that at most one of these maximal subgroups is abelian, and in particular $G$ is not minimal non-abelian.
\end{proof}

We are now able to prove our second main theorem.

\begin{proof}[Proof of Theorem \ref{thm:pgpnc}]
Lemma \ref{lem:zinfrat} shows that if $G$ is non-abelian, $2$-generated and not minimal non-abelian, then $G$ contains at most one abelian maximal subgroup, and $Z(G) < \Phi(G)$. Hence, in this case, each maximal subgroup of $G$ contains $Z(G)$. Therefore, exactly one of (i), (ii) and (iii) holds.

\medskip

\noindent (i) This is clear.

\medskip

\noindent (ii) This is an immediate consequence of Corollary \ref{cor:highgengp}.

\medskip

\noindent (iii)(a) By Lemma \ref{lem:zinfrat}\ref{zinfrat2}, the abelian maximal subgroup $M$ of $G$ is the unique maximal subgroup containing each element of $M \setminus \Phi(G)$. It follows from Proposition \ref{prop:isolvertnc} that a vertex $g$ of $\nc(G)$ is isolated if and only if $g \in M \setminus \Phi(G)$. In particular, $\nc(G)$ has isolated vertices, and hence $\nd(G)$ is connected with diameter $2$ by Theorem \ref{thm:nilpncsummary}.

\medskip

\noindent (iii)(b) Let $x, y \in G \setminus Z(G)$, and let $L$ and $M$ be maximal subgroups of $G$ that contain $x$ and $y$, respectively. As $Z(L) = Z(G) = Z(M)$, applying Lemma \ref{lem:ncmaxnorm} to the $4$-tuple $(x,L,y,M)$ gives $d(x,y) \le 2$. It follows from Proposition \ref{prop:nodiamonenc} that $\nc(G)$ is connected with diameter $2$.

\medskip

\noindent (iii)(c) Let $M$ be a (non-abelian) maximal subgroup of $G$ that satisfies $Z(G) < Z(M)$. In addition, let $y \in M \setminus \Phi(G)$, and let $x \in Z(M) \setminus Z(G)$. Proposition \ref{prop:normisol} implies that each element of $Z(M)$ lies in a maximal subgroup of $G$ distinct from $M$. Lemma \ref{lem:zinfrat}\ref{zinfrat2} shows that $M$ is the unique maximal subgroup of $G$ containing $y$, and thus $y \notin Z(M)$. Additionally, the element $x$ of $Z(M)$ lies in a maximal subgroup $L$ of $G$ that is not equal to $M$. Since $M = C_G(x)$, it follows that $x \notin Z(L)$. Hence applying Lemma \ref{lem:ncmaxnorm} to the $4$-tuple $(x,L,y,M)$ yields $d(x,y) = 3$. Theorem \ref{thm:nilpncsummary} therefore implies that $\nc(G)$ is connected with diameter $3$.
\end{proof}

\section{Direct products of groups}
\label{sec:dirprods}

In this section, we explore the structure of the non-commuting, non-generating graph of a direct product of groups. In particular, we conclude with a proof of Theorem \ref{thm:nilpgpnc}.

Let $G$ and $H$ be arbitrary groups, with $g_1, g_2 \in G$ and $h_1, h_2 \in H$. Observe that if $[g_1,g_2] \ne 1$ or $[h_1,h_2] \ne 1$, then the elements $(g_1,h_1)$ and $(g_2,h_2)$ of $G \times H$ do not commute and are therefore non-central. In particular, these elements are vertices of $\nc(G \times H)$. Additionally, if $\langle g_1, g_2 \rangle \ne G$ or $\langle h_1,h_2 \rangle \ne H$, then $\langle (g_1,h_1),(g_2,h_2)\rangle \ne G \times H$. Therefore, relatively weak conditions are required for two vertices of $\nc(G \times H)$ to be adjacent. We utilise this fact to prove the following two results.

\begin{lem}
\label{lem:direc2gen}
Suppose that $G$ is non-abelian and $H$ is non-cyclic. Then $\nc(G \times H)$ is connected with diameter $2$.
\end{lem}

\begin{proof}
Let $(g_1,h_1),(g_2,h_2) \in (G \times H)\setminus(Z(G \times H))$. Observe that if $g_1, g_2 \in Z(G)$, then $h_1, h_2 \notin Z(H)$, and vice versa. We split the proof into three (not all mutually exclusive) cases. By Proposition \ref{prop:nodiamonenc}, it suffices in each case to find a path of length $2$ in $\nc(G)$ between $(g_1,h_1)$ and $(g_2,h_2)$.

\medskip

\noindent \textbf{Case (a)}: $g_1, g_2 \notin Z(G)$. Proposition \ref{prop:noncomdiam} shows that there exists an element $u \in G$ that does not commute with $g_1$ or $g_2$. Additionally, as $H$ is not cyclic, $\langle h_1,1 \rangle < H$ and $\langle 1,h_2 \rangle < H$. It follows that $((g_1,h_1),(u,1),(g_2,h_2))$ is a path in $\nc(G \times H)$.

\medskip

\noindent \textbf{Case (b)}: $h_1, h_2 \notin Z(H)$. This case only occurs if $H$ is non-abelian, and we conclude as in Case (a) that $d((g_1,h_1),(g_2,h_2)) \le 2$.

\medskip

\noindent \textbf{Case (c)}: Exactly one of $g_1$ and $g_2$ is central in $G$, and exactly one of $h_1$ and $h_2$ is central in $H$. We will assume without loss of generality that $h_1 \in Z(H)$, so that $h_2 \notin Z(H)$. Then, since $(g_1,h_1),(g_2,h_2) \notin Z(G \times H)$, we deduce that $g_1 \notin Z(G)$ and $g_2 \in Z(G)$. Let $s \in G \setminus C_G(g_1)$ and $t \in H \setminus C_G(h_2)$. The non-abelian group $G$ properly contains the abelian group $\langle g_2,s \rangle$, and similarly, $\langle h_1, t \rangle < H$. Thus $((g_1,h_1),(s,t),(g_2,h_2))$ is a path in $\nc(G \times H)$.
\end{proof}

The above result shows that the non-commuting, non-generating graph behaves very differently to the generating graph in the context of $2$-generated direct powers of non-abelian groups. Indeed, as mentioned in \S\ref{sec:intro}, Crestani and Lucchini proved that there is no upper bound for the diameter of $\Gamma^+(G)$, for $G$ such a direct power.

In the following proposition and its proof, we use the symbol $d$ to denote distances in both $\nc(G)$ and $\nc(G \times H)$. In each case, it will be clear which graph contains the relevant vertices. Additionally, if $x$ and $y$ are vertices of a graph that lie in distinct connected components, then we write $d(x,y) = \infty$.

\begin{prop}
\label{prop:direccyclic}
Suppose that $G$ is non-abelian and $H$ is cyclic. Additionally, let $g_1$ and $g_2$ be distinct elements of $G \setminus Z(G)$, and let $h_1, h_2 \in H$. Then the following statements hold.
\begin{enumerate}[label={(\roman*)},font=\upshape]
\item $d((g_1,h_1),(g_2,h_2)) \le d(g_1,g_2)$. \label{direccycl1}
\item If $g_1$ is not isolated in $\nc(G)$, then $d((g_1,h_1),(g_1,h_2)) \in \{0,2\}$. \label{direccycl2}
\end{enumerate}
Hence if $\nc(G)$ is connected with diameter $k$, then $\nc(G \times H)$ is connected with diameter at most $k$.
\end{prop}

\begin{proof}
Suppose that $r:=(x_1,x_2,\ldots,x_k)$ is a path in $\nc(G)$. As no two adjacent vertices in this path commute or generate $G$, it follows that $((x_1,h_1),(x_2,h_2),\ldots,(x_k,h_2))$ is a path in $\nc(G \times H)$.

\medskip

\noindent (i) This is clear if $d(g_1,g_2) = \infty$. Otherwise, we obtain the result by setting $r$ to be a path in $\nc(G)$ of minimal length with $x_1 = g_1$ and $x_k = g_2$.

\medskip

\noindent (ii) This is clear if $h_1 = h_2$. Otherwise, note that $(g_1,h_1)$ and $(g_1,h_2)$ commute, and hence are not adjacent in $\nc(G \times H)$. The result now follows by setting $r$ to be the path $(g_1,x_2,g_1)$, with $x_2$ some neighbour of $g_1$ in $\nc(G)$.

Suppose finally that $\nc(G)$ is connected with diameter $k$, and recall from Proposition \ref{prop:nodiamonenc} that $k \ge 2$. Since $\{(g,h) \mid g \in G \setminus Z(G), h \in H\}$ is the set of vertices of $\nc(G \times H)$, it follows from (i) and (ii) that $\nc(G \times H)$ is connected with diameter at most $k$.
\end{proof}


It is easy to show, using Burnside's Basis Theorem, that the direct product of a finite $d_1$-generated $p$-group and a finite $d_2$-generated $p$-group is $(d_1+d_2)$-generated. Thus no non-abelian $2$-generated finite $p$-group can be expressed as a nontrivial direct product. It will follow from Theorem \ref{thm:nilpgpnc} that if $H$ is a finite cyclic group and $G$ is a finite nilpotent group such that $\nc(G)$ is connected, then $\diam(\nc(G \times H)) \ne \diam(\nc(G))$ if and only if {$\diam(\nc(G \times H)) = 3$} and $|G|$ and $|H|$ have a common prime divisor. However, this is not always the case if $G$ is not nilpotent. For example, the non-commuting, non-generating graph of the symmetric group $S_4$ is connected with diameter $3$, as is the graph $\nc(S_4 \times C_3)$. On the other hand, $\diam(\nc(S_4 \times C_2)) = 2$ (note that $S_4 \times C_2$ is $2$-generated).

We now wish to determine which vertices of $\nc(G \times H)$ are isolated. Proposition \ref{prop:isolvertnc} suggests that a classification of the maximal subgroups of $G \times H$ will aid us in this task. The next result, which is a consequence of Goursat's Lemma \cite[\S11--12]{goursat}, gives such a classification.

\begin{lem}[{\cite[p.~354]{thevenaz}}]
\label{lem:maxsubgdirect}
Let $G$ and $H$ be groups, and let $K$ be a subgroup of $G \times H$. Then $K$ is a maximal subgroup of $G \times H$ if and only if one of the following occurs:
\begin{enumerate}[label={(\roman*)},font=\upshape]
\item $K = M_G \times H$, for some maximal subgroup $M_G$ of $G$;
\item $K = G \times M_H$, for some maximal subgroup $M_H$ of $H$; or
\item $K = \{(g,h) \mid g \in G, h \in H, (N_1g)\theta = N_2h\}$, where $N_1$ and $N_2$ are maximal normal subgroups of $G$ and $H$, respectively, and $\theta$ is an isomorphism from $G/N_1$ to $H/N_2$.
\end{enumerate}
\end{lem}

In what follows, we assume the convention that if $H$ is an infinite group, then each positive integer divides $|H|$.

\begin{thm}
\label{thm:direcunmax}
Let $G$ and $H$ be finitely generated groups, with $G$ non-cyclic, and let $g \in G$ and $h \in H$. Additionally, let $\mathcal{L}$ be the set of maximal subgroups $L$ of $G$ such that $L \trianglelefteq G$ and $|G:L|$ divides $|H|$. Then $(g,h)$ lies in a unique maximal subgroup of $G \times H$ if and only if all of the following hold:
\begin{enumerate}[label={(\roman*)},font=\upshape]
\item \label{direcunmax1} $\langle h \rangle = H$;
\item \label{direcunmax2} $g$ lies in a unique maximal subgroup $M_G$ of $G$; and
\item \label{direcunmax3} $\mathcal{L} \subseteq \{M_G\}$.
\end{enumerate}
\end{thm}

\begin{proof}
Since $G$ is finitely generated and $\langle g \rangle \ne G$, there exists a maximal subgroup $M_G$ of $G$ containing $g$. Similarly, there exists a maximal subgroup $M_H$ of $H$ containing $h$ if and only if $\langle h \rangle \ne H$. Notice that $(g,h)$ lies in the maximal subgroups $M_G \times H$ and $G \times M_H$ of $G \times H$ for every such $M_G$ and $M_H$. Hence if $(g,h)$ lies in a unique maximal subgroup of $G \times H$, then (i) and (ii) hold. For the remainder of the proof, we will assume that these conditions hold.

Lemma \ref{lem:maxsubgdirect} shows that $(g,h)$ lies in a maximal subgroup of $G \times H$ other than $M_G \times H$ if and only if there exist maximal normal subgroups $N_1$ of $G$ and $N_2$ of $H$, and an isomorphism $\theta: G/N_1 \to H/N_2$ with $(N_1g)\theta = N_2h$. If this is the case, then since $\langle h \rangle = H$, we conclude that $\langle N_2h \rangle = H/N_2$, and hence $\langle N_1g \rangle = G/N_1$.

Now, the quotients of $H$ by its maximal normal subgroups are exactly the cyclic groups of order a prime dividing $|H|$. Hence such a normal subgroup $N_1$ of $G$ exists if and only if $N_1 \in \mathcal{L}$ and $g \notin N_1$, i.e., if and only if $\mathcal{L}$ contains a subgroup not equal to $M_G$. Therefore, $M_G \times H$ is the unique maximal subgroup of $G \times H$ containing $(g,h)$ if and only if (iii) holds, as required.
\end{proof}



We are now able to describe the isolated vertices of $\nc(G \times H)$.

\begin{cor}
\label{cor:direcisol}
Let $G$ and $H$ be groups, with $G$ non-abelian, and let $g \in G$ and $h \in H$. Additionally, let $\mathcal{L}$ be the set of maximal subgroups $L$ of $G$ such that $L \trianglelefteq G$ and $|G:L|$ divides $|H|$. Then $(g,h)$ is an isolated vertex of $\nc(G \times H)$ if and only if all of the following hold:
\begin{enumerate}[label={(\roman*)},font=\upshape]
\item $\langle h \rangle = H$; \label{direcisol1}
\item $g$ is an isolated vertex of $\nc(G)$; and \label{direcisol2}
\item \label{direcisol3} $|\mathcal{L}| \le 1$, and if $\mathcal{L} = \{L\}$, then $g \in L$.
\end{enumerate}
\end{cor}

\begin{proof}
Recall from Corollary \ref{cor:highgengp} that the non-commuting, non-generating graph of a group that is not $2$-generated has no isolated vertices. Hence if $G$ is not $2$-generated, then (ii) does not hold, and if $H$ is not $2$-generated, then (i) does not hold. In either of these cases, $G \times H$ is not $2$-generated, and hence $\nc(G \times H)$ has no isolated vertices. We may therefore assume that $G$ and $H$ are $2$-generated.

Suppose first that $(g,h)$ is an isolated vertex of $\nc(G \times H)$. Then Proposition \ref{prop:isolvertnc} implies that $G \times H$ has a unique maximal subgroup $K$ containing $(g,h)$, and $(g,h) \in Z(K)$. Hence by Theorem \ref{thm:direcunmax}, (i) and (iii) hold, and $g$ lies in a unique maximal subgroup $M_G$ of $G$. Since $(g,h)$ lies in the maximal subgroup $M_G \times H$ of $G \times H$, we conclude that $K = M_G \times H$. In addition, since $(g,h) \in Z(K)$, we deduce that $g \in Z(M_G)$. Proposition \ref{prop:isolvertnc} therefore yields (ii).

Conversely, suppose that (i), (ii) and (iii) all hold. Then it follows from Proposition \ref{prop:isolvertnc} that $g$ lies in a unique maximal subgroup $M_G$ of $G$, and $g \in Z(M_G)$. Hence Theorem \ref{thm:direcunmax} implies that $M_G \times H$ is the unique maximal subgroup of $G \times H$ containing $(g,h)$. Since $(g,h) \in Z(M_G \times H)$, the vertex $(g,h)$ is isolated in $\nc(G)$ by Proposition \ref{prop:isolvertnc}.
\end{proof}

\begin{proof}[Proof of Theorem \ref{thm:nilpgpnc}]

Observe that exactly one of (i), (ii) and (iii) holds. We may therefore consider each case separately.

\medskip

\noindent (i) This is clear.

\medskip

\noindent (ii) This is an immediate consequence of Lemma \ref{lem:direc2gen}.

\medskip

\noindent (iii)(a) As $\nc(P)$ is not the empty graph, Theorem \ref{thm:nilpncsummary} implies that $\nc(P)$ is connected with diameter $2$ or $3$. Hence Propositions \ref{prop:nodiamonenc} and \ref{prop:direccyclic} show that $\nc(G)$ is connected, with $1 < \diam(\nc(G)) \le \diam(\nc(P))$. It therefore suffices to find a pair of vertices of $\nc(G)$ of distance $3$ in the case $\diam(\nc(P)) = 3$. Here, Theorem \ref{thm:pgpnc} shows that $P$ is $2$-generated, and that there exists a non-abelian maximal subgroup $M$ of $P$ with $Z(P) < Z(M)$. As $|P|$ and $|H|$ are coprime, $G$ is also $2$-generated.

Let $a \in M \setminus (\Phi(P) \cup Z(M))$, $b \in Z(M) \setminus Z(P)$ and $x \in H$, with $\langle x \rangle = H$. By Lemma \ref{lem:zinfrat}\ref{zinfrat2}, $M$ is the unique maximal subgroup of $P$ containing $a$. As $|P|$ is coprime to $|H|$, the set $\mathcal{L}$ in Theorem \ref{thm:direcunmax} is empty, and so $M \times H$ is the unique maximal subgroup of $G$ containing $(a,x)$. Additionally, $(b,x)$ lies in $Z(M \times H)$, while $(a,x)$ does not. Since $\nc(P)$ contains no isolated vertices, we conclude from Proposition \ref{prop:isolvertnc} that there exists a maximal subgroup $L$ of $P$ with $b \in L \setminus Z(L)$, and hence $(b,x) \in (L \times H) \setminus Z(L \times H)$. Applying Lemma \ref{lem:ncmaxnorm} to the $4$-tuple $((b,x),L \times H,(a,x),M \times H)$ therefore yields $d((a,x),(b,x)) = 3$.

\medskip

\noindent (iii)(b) Since $|P|$ is coprime to $|H|$, the set $\mathcal{L}$ in Corollary \ref{cor:direcisol} is empty. Hence this corollary shows that the set of isolated vertices of $\nc(G)$ is as required. In particular, $\nc(G) \ne \nd(G)$ and $\nd(G)$ is not empty (note that if $x \in P \setminus Z(P)$, then $(x,1)$ is a vertex of $\nd(G)$). Theorem \ref{thm:nilpncsummary} therefore implies that $\nd(G)$ is connected with diameter $2$.
\end{proof}

\subsection*{Acknowledgements}
The authors would like to thank the Isaac Newton Institute for Mathematical Sciences for support and hospitality during the programme ``Groups, representations and applications: new perspectives'', when work on this paper was undertaken. This work was supported by EPSRC grant number EP/R014604/1. In addition, this work was partially supported by a grant from the Simons Foundation. Finally, the second author was supported by a St Leonard's International Doctoral Fees Scholarship and a School of Mathematics \& Statistics PhD Funding Scholarship at the University of St Andrews.

\bibliographystyle{plain}
\bibliography{Nilprefs}

\end{document}